\documentclass{agtart_a}
\pdfoutput=1

\usepackage{pinlabel}
\usepackage{graphicx}
\usepackage{color}

\title{Geodesic knots in cusped hyperbolic 3--manifolds}

\author{Sally M Kuhlmann}
\givenname{Sally M}
\surname{Kuhlmann}
\address{Department of Mathematics and Statistics\\
University of Melbourne\\\newline
Victoria, 3010\\Australia}
\email{S.Kuhlmann@ms.unimelb.edu.au}
\urladdr{} 

\volumenumber{6}
\issuenumber{}
\publicationyear{2006}
\papernumber{74}
\startpage{2151}
\endpage{2162}

\doi{}
\MR{}
\Zbl{}

\keyword{simple closed geodesic}
\keyword{knot}
\keyword{hyperbolic 3-manifold}
\subject{primary}{msc2000}{57N10}
\subject{primary}{msc2000}{53C22}
\subject{secondary}{msc2000}{57M50}
\subject{secondary}{msc2000}{30F40}

\received{21 April 2006}
\revised{}
\accepted{5 September 2006}
\published{19 November 2006}
\publishedonline{19 November 2006}
\proposed{}
\seconded{}
\corresponding{}
\editor{}
\version{}

\arxivreference{}

\AtBeginDocument{\let\bar\wbar\let\tilde\wtilde\let\hat\what}
\AtBeginDocument{\let\widehat\wwhat\let\widetilde\wwtilde}

\def\N{\mathbb N}
\def\H{\mathbb H}
\newcommand{\dist}{\operatorname{dist}}
\newcommand{\PSL}{\ensuremath{\mathrm{PSL}}}
\newcommand{\Isom}{\ensuremath{\mathrm{Isom}}}

\makeatletter
\def\cnewtheorem#1[#2]#3{\newtheorem{#1}{#3}[section]
\expandafter\let\csname c@#1\endcsname\c@thm}
\makeatother

\theoremstyle{plain}
\newtheorem{thm}{Theorem}[section]
\cnewtheorem{lem}[thm]{Lemma}
\newtheorem*{cuspedthm}{\fullref{thminfinite}}

\theoremstyle{definition}
\newtheorem*{defn}{Definition}
\newtheorem{que}{Question}

\begin{document}

\begin{asciiabstract}
We consider the existence of simple closed geodesics or "geodesic knots" in 
finite volume orientable hyperbolic 3-manifolds.  Previous results show that at
least one geodesic knot always exists [Bull. London Math. Soc. 31(1) (1999)
81-86], and that certain arithmetic manifolds contain infinitely many geodesic
knots [J. Diff. Geom. 38 (1993) 545-558], [Experimental Mathematics 10(3) 
(2001) 419-436].  In this paper we show that all cusped orientable finite 
volume hyperbolic 3-manifolds contain infinitely many geodesic knots.  Our 
proof is constructive, and the infinite family of geodesic knots produced 
approach a limiting infinite simple geodesic in the manifold.
\end{asciiabstract}

\begin{webabstract}
We consider the existence of simple closed geodesics or ``geodesic knots'' in 
finite volume orientable hyperbolic 3-manifolds.  Previous results show that a
least one geodesic knot always exists [Bull. London Math. Soc. 31(1) (1999)
81--86], and that certain arithmetic manifolds contain infinitely many geodesic
knots [J. Diff. Geom. 38 (1993) 545--558], [Experimental Mathematics 10(3) 
(2001) 419--436].  In this paper we show that all cusped orientable finite 
volume hyperbolic 3-manifolds contain infinitely many geodesic knots.  Our 
proof is constructive, and the infinite family of geodesic knots produced 
approach a limiting infinite simple geodesic in the manifold.
\end{webabstract}

\begin{abstract}
We consider the existence of simple closed geodesics or ``geodesic knots'' in 
finite volume orientable hyperbolic 3--manifolds.  Previous results show that 
at least one geodesic knot always exists~\cite{ahs}, and that certain 
arithmetic manifolds contain infinitely many geodesic knots~\cite{chr}, 
\cite{me}.  In this paper we show that all cusped orientable finite volume 
hyperbolic 3--manifolds contain infinitely many geodesic knots.  Our proof is 
constructive, and the infinite family of geodesic knots produced 
approach a limiting infinite simple geodesic in the manifold.
\end{abstract}

\maketitle

\section{Introduction}

A geodesic in a Riemannian manifold is said to be simple if it has no 
self-intersections and nonsimple otherwise.  In this paper we study geodesics 
in hyperbolic manifolds, that is, complete Riemannian manifolds of constant 
curvature $-1$.  Properties of simple closed geodesics in hyperbolic 
3--manifolds were first studied by Sakai~\cite{s}, who introduced the 
terminology ``geodesic knots'' to describe them.
\vspace{-2pt}

A Riemannian manifold may or may not contain simple closed geodesics, and the 
question of which ones do is open in general.  However, the answer is known for
hyperbolic manifolds of dimension two and three.  In an orientable finite area 
hyperbolic 2--manifold, each non-contractible non-peripheral simple closed 
curve is homotopic to a simple closed geodesic, hence the only example not 
containing a simple closed geodesic is the thrice-punctured sphere.  The case 
of orientable hyperbolic 3--manifolds was solved by Adams, Hass and 
Scott~\cite{ahs}.  They showed that every finite volume orientable hyperbolic
3--manifold contains a geodesic knot, and that the only non-elementary infinite
volume exception is the quotient of $\H^3$ by a Fuchsian group corresponding to
the thrice-punctured sphere.
\vspace{-2pt}

Since self-intersections of a 1--dimensional loop in a 3--dimensional manifold
should be rare, we might expect that most hyperbolic 3--manifolds in fact admit
infinitely many geodesic knots.  This paper thus addresses the question:
\begin{que} \label{queinfinite}
Which hyperbolic 3--manifolds contain infinitely many geodesic knots?
\end{que}


Some partial answers to \fullref{queinfinite} are known.
Chinburg and Reid~\cite{chr} showed that there exist infinitely many 
noncommensurable closed hyperbolic 3--manifolds all of whose closed geodesics 
are simple, thus providing examples of arithmetic hyperbolic 3--manifolds
containing infinitely many geodesic knots.
In \cite{me} we showed using arithmetic methods that if $M=\H^3/\Gamma$ is a 
hyperbolic 3--orbifold such that $\Gamma$ is a finite index subgroup of the 
Bianchi group $\Gamma_d=\PSL_2(\mathcal{O}_d)$ for some square-free positive 
integer $d$, then $M$ contains infinitely many geodesic knots.
In this paper we take a geometric approach and prove a general result for 
cusped hyperbolic 3--manifolds, that is, non-compact hyperbolic 3--manifolds of
finite volume.

\begin{thm}\label{thminfinite}
  Every cusped orientable hyperbolic 3--manifold contains infinitely
  many geodesic knots.
\end{thm}

The idea is to consider infinite families of closed geodesics which approach a 
limiting infinite simple geodesic in the manifold, and to show that infinitely 
many of these closed geodesics are embedded.  This idea can be illustrated one 
dimension down.  Consider an infinite cusp-to-cusp geodesic in a cusped 
hyperbolic 2--manifold.  This geodesic can be perturbed slightly to form a 
closed geodesic which spirals some distance towards the cusp before spiralling 
back and closing up, as in \fullref{figspiral}.  By varying the amount of 
spiralling, we can obtain infinitely many such closed geodesics.  Moreover, 
with the extra dimension in three-dimensional cusped manifolds, these geodesics
should also typically avoid self-intersection.
\begin{figure}[ht!]
  \begin{center}
    \includegraphics[width=\textwidth]{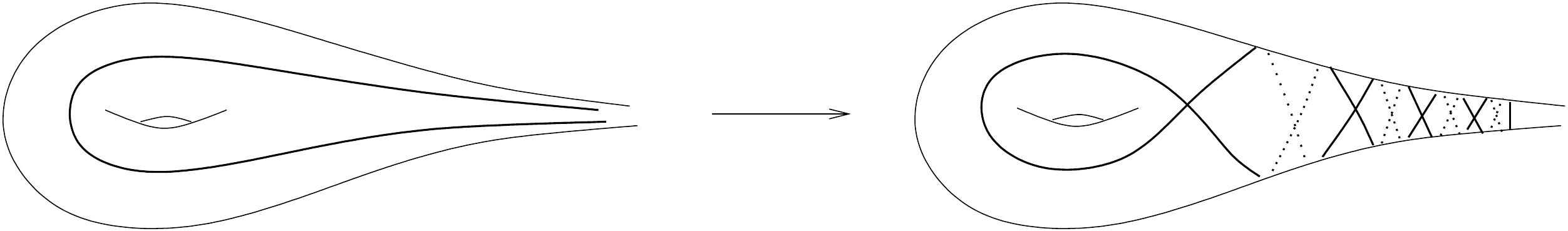}
  \end{center}
  \caption[A spiralling closed geodesic in a hyperbolic 2--manifold]
  {A cusp-to-cusp geodesic can be perturbed slightly to form a
closed curve, spiralling some way out into the cusp.}
  \label{figspiral}
\end{figure}

The work in this paper was part of the PhD thesis~\cite{phd}, where an 
existence result for geodesic knots in closed hyperbolic 3--manifolds is also 
proved.  In particular it is shown that if a closed hyperbolic 3--manifold 
satisfies certain geometric and arithmetic conditions, then it contains 
infinitely many geodesic knots.  The conditions on the manifold are checkable, 
and have been verified for many manifolds in the Hodgson--Weeks closed census.
This result will appear in the subsequent paper~\cite{closed}.

This was partially supported by an Australian Postgraduate Award.
I wish to thank Craig Hodgson for his suggestions and help.


\section{An infinite family of closed geodesics}
\label{secgammapq}
\vspace{4pt}

To prove \fullref{thminfinite}, we start by defining a suitable class of 
spiralling closed geodesics in dimension three.  Let our cusped hyperbolic 
3--manifold be $M=\H^3/\Gamma$ for a discrete torsion-free group of isometries
$\Gamma\subset\Isom^+(\H^3)$.  Choose a cusp $C$ of $M$ and expand a horoball 
neighbourhood of it with torus boundary.  The maximal embedded horoball 
neighbourhood $U$ has boundary $T$ a torus with a finite number of 
self-tangencies.  Pick a point of self-tangency, a \emph{bumping point} $A$;
then let $\gamma_\infty$ be the infinite geodesic passing from $C$ to itself 
through $A$, perpendicular to $T$, as shown schematically in 
\fullref{figgammainfty}.
\vspace{4pt}
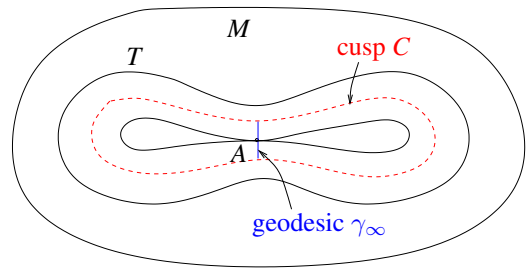
\begin{figure}[ht!]
  \begin{center}
      \labellist
        \small\hair 2pt
        \pinlabel $M$ at 190 200
        \pinlabel $T$ at 103 177
        \pinlabel { {\color{red} cusp $C$} } at 305 183
        \pinlabel $A$ at 190 96
        \pinlabel { {\color{blue} geodesic $\gamma_\infty$} } at 260 36
      \endlabellist
      \includegraphics[height=3.5cm]{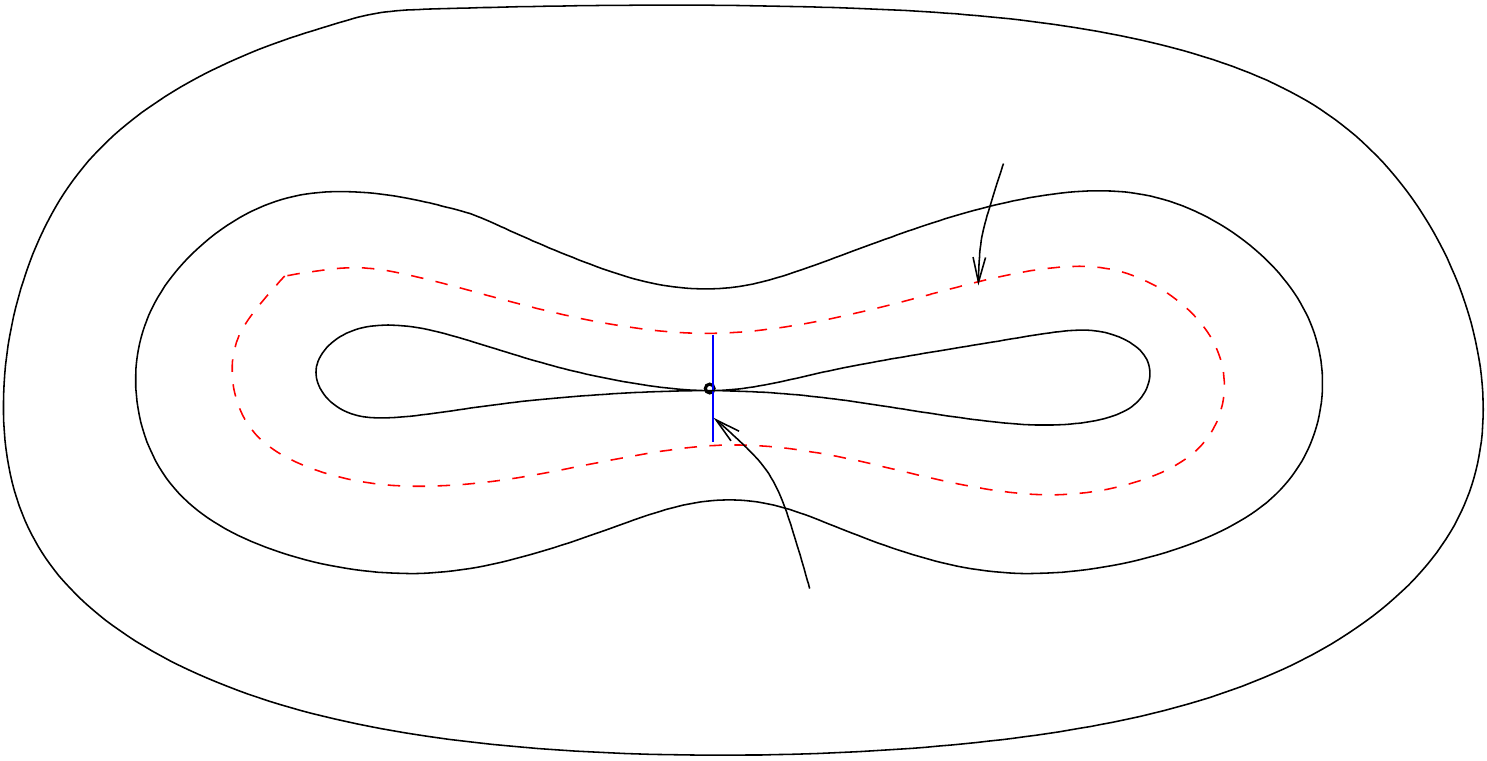}
  \end{center}
  \caption{The infinite cusp-to-cusp geodesic $\gamma_\infty$}
  \label{figgammainfty}
\end{figure}
\vspace{4pt}

Consider a lift $\widetilde{A}\in\H^3$ of the bumping point $A$ under
the covering projection $\pi\co\H^3=\widetilde{M}\to M$.
Then the preimage $\pi^{-1}(T)$ of $T$ contains two horospheres say
$H$ and $H'$ which intersect precisely at $\widetilde{A}$.  See
\fullref{figadjacent}.
\vspace{4pt}
\begin{figure}[ht!]
  \begin{center}
     \labellist
       \small\hair 2pt
       \pinlabel $H$ at 55 137
       \pinlabel $H'$ at 180 160
       \pinlabel $\wtilde{A}$ at 157 83
     \endlabellist
     \includegraphics[height=3.5cm]{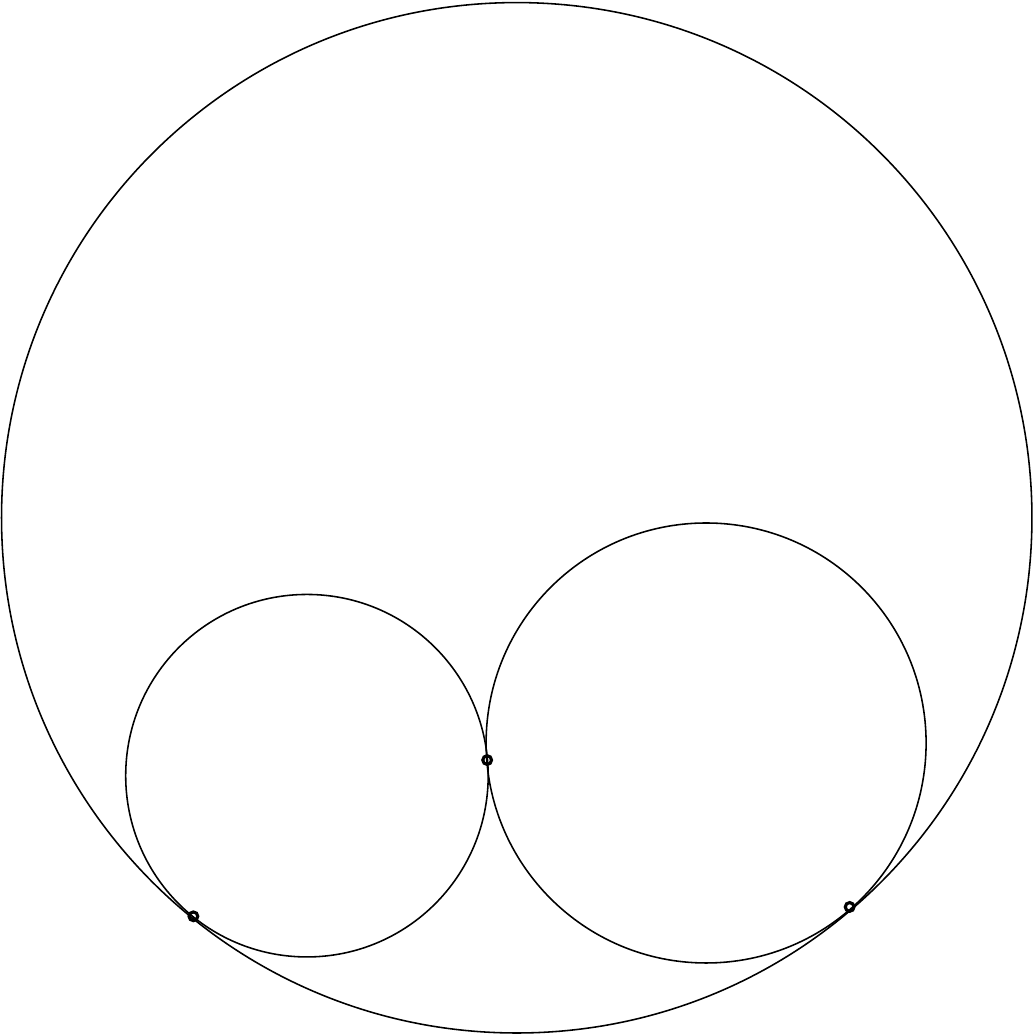}
  \end{center}
  \caption{Adjacent horospheres $H$ and $H'$ in the Poincar\'e ball
model of $\H^3$}
  \label{figadjacent}
\end{figure}
\vspace{4pt}
Since they both cover $T$, $H$ and $H'$ are related by certain       
isometries in $\Gamma$.       
The closed geodesics that we consider in $M$ are those       
corresponding to all possible isometries in $\Gamma$ mapping $H$ to $H'$.
\vspace{4pt}

Let $g\in\Gamma$ be an arbitrary such isometry taking $H$ to $H'$,
and choose generators $a,b\in\Gamma$ for the cusp group of $C$
which fix $H'$ in $\H^3$.
Then the isometries in $\Gamma$ mapping $H$ to $H'$
are precisely those composed of
\begin{itemize}
\item an isometry fixing $H$, followed by
\item any given isometry mapping $H$ to $H'$, followed by
\item an isometry fixing $H'$.
\end{itemize}
This yields an element of the form
\begin{align*}
a^{p_1}b^{q_1}\cdot g\cdot (g^{-1}ag)^{p_2}(g^{-1}bg)^{q_2}
&=a^{p_1}b^{q_1}a^{p_2}b^{q_2}g &&\text{for $p_i,q_i\in\Z$}\\
&=a^pb^qg &&\text{for $p,q\in\Z$}\\
&=:g_{p,q}
\end{align*}
Denote the geodesic axis in $\H^3$ corresponding to $g_{p,q}$ by
$\widetilde{\gamma}_{p,q}$, and the closed geodesic in $M$ by
$\gamma_{p,q}$.
So for any given $g$, $a$ and $b$, the set of geodesics corresponding to all 
possible isometries mapping $H$ to $H'$ can be parametrised by $\Z\times\Z$.
Later we will make a canonical choice for $g$, given elements $a$ and $b$ 
generating the cusp group, and 
will see that as $|(p,q)|\to\infty$, the geodesic axes 
$\widetilde{\gamma}_{p,q}$ approach a lift of the infinite geodesic 
$\gamma_\infty$.
This limiting behaviour is in fact independent of the choice of $g$, $a$ and 
$b$.


\section{A normalised preimage in $\H^3$}
\label{secnormalised}

To study the limiting behaviour of the geodesics
$\gamma_{p,q}$ and hence show that infinitely many are simple, we
normalise the position of the preimage of $T$ in $\H^3$. To do this
we must first make more precise our notation for the generators $a,b$
of the cusp group, and their corresponding loops in $M$.

Let $\widehat{T}_r$ denote the embedded horospherical torus around $C$
which is hyperbolic distance $r$ further into the cusp than $T$.
Then $a$ and $b$ correspond to elements of $\pi_1(\widehat{T}_r)$.
Pick an $r>0$.  Then there is a projection map
$\hat{\pi}\co\widehat{T}_r\to T$ mapping points radially from the cusp,
so that two points $\widehat{A}$ and $\widehat{B}$ on the nonsingular torus
$\widehat{T}_r$ are mapped to the singular point $A$ of $T$,
as shown in \fullref{figtorusproj}.
\begin{figure}[ht!]
  \begin{center}
    \labellist
      \small\hair 2pt
      \pinlabel $r$ at 185 120
      \pinlabel $\widehat{B}$ at 150 69
      \pinlabel $\widehat{A}$ at 125 37
      \pinlabel $\widehat{T}_r$ at 250 80
      \pinlabel $\what{\pi}$ at 362 67
      \pinlabel $A$ at 575 65
      \pinlabel $T$ at 690 80
    \endlabellist
    \includegraphics[width=\textwidth]{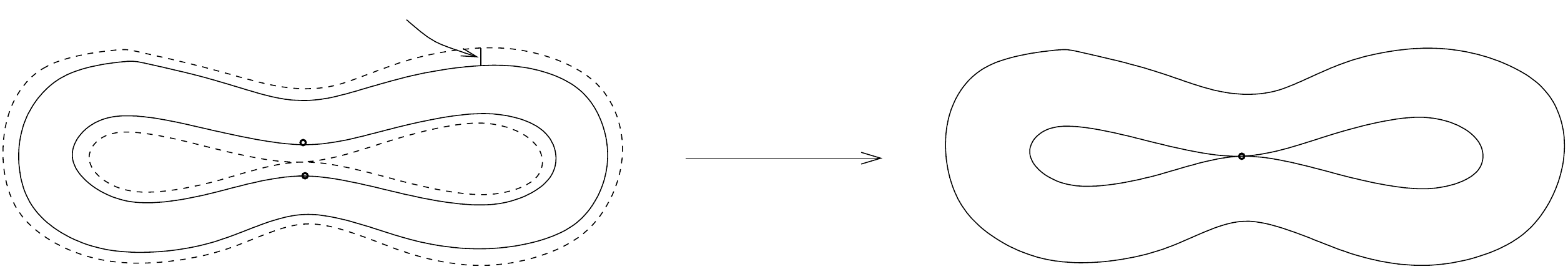}
  \end{center}
  \caption[The projection map $\hat{\pi}\co\widehat{T_r}\to T$,  
    mapping two points $\widehat{A}$ and $\widehat{B}$ to $A$.]
{The projection map $\hat{\pi}\co\widehat{T}_r\to T$
  maps two points $\widehat{A}$ and $\widehat{B}$ to $A$.}
  \label{figtorusproj}
\end{figure}
Let $\widehat{\alpha},\widehat{\beta}\in\pi_1(\widehat{T}_r)$ be
loops on $\widehat{T}_r$ representing $a$ and $b$ respectively.
Since $a$ and $b$ generate the cusp group,
$\widehat{\alpha}$ and $\widehat{\beta}$ are dual oriented
non-contractible simple closed curves on $\widehat{T}_r$ which
after isotopy we may assume are
Euclidean geodesics with common basepoint $\widehat{A}$.
We will generally identify these curves with their image curves
$\alpha:=\hat{\pi}(\widehat{\alpha})$ and              
$\beta:=\hat{\pi}(\widehat{\beta})$ on $T$,
pulling back to $\widehat{T}_r$ to avoid complications arising from the
self-tangencies of $T$.

We now work in the upper half-space model of $\H^3$, and regard 
$\Gamma\cong\pi_1(M)$ as a subgroup of $\PSL_2(\C)$.
Let $\C_\infty=\C\cup\{\infty\}$ denote the sphere at infinity of the upper 
half-space model of $\H^3$, and let the horosphere centred at $z\in\C_\infty$ 
be $H_z$.  The preimage $\pi^{-1}(T)$ of $T$ is a union of horospheres
meeting tangentially at the preimage of the singular point(s) of $T$.
By an isometry, normalise so that $H'$ is the horizontal
horosphere $H_\infty$ at height 1, and $H$ is the horosphere $H_0$,
tangent to $H'$.
Parametrise $H_\infty$ by $\C$ to match with $\C_\infty$ under vertical
translation.
\vspace{4pt}

The preimages of the curves $\alpha,\beta\subset T$ determine a
tiling of each horosphere $H_z\subset\pi^{-1}(T)$ by parallelograms.
Each parallelogram contains two points covering the bumping point
$A\in M$, corresponding to $\widehat{A}$ (at the vertex of the
parallelogram) and $\widehat{B}$ in $\widehat{T}_r$.
On $H_\infty$ we see this as a tiling of the complex plane.
Let $t_\alpha$ and $t_\beta$ be the complex numbers corresponding to
translation along lifts of the curves $\alpha$ and $\beta$ on
$H_\infty$.
Thus the parabolic elements $a,b\in\Gamma$ have representations
\begin{displaymath}
a=\begin{pmatrix} 1 & t_\alpha \\ 0 & 1 \end{pmatrix}
\qquad\text{and}\qquad
b=\begin{pmatrix} 1 & t_\beta \\ 0 & 1 \end{pmatrix}\ .
\end{displaymath}
Let the lift of $A$ at complex coordinate 0 on $H_\infty$ be
$A_{0,0}$, and
let the other lift of $A$ in the parallelogram
$\{rt_\alpha+st_\beta:r,s\in[0,1)\}\subset H_\infty$ be $B_{0,0}$.
For $(p,q)\in\Z^2$
let $A_{p,q}$ be the lift of $A$ on $H_\infty$ obtained from
$A_{0,0}$ by translation by the complex number
$a_{p,q}:=pt_\alpha+qt_\beta$,
and define $B_{p,q}$ and $b_{p,q}$ similarly.
So the horosphere $H_{a_{p,q}}\subset\pi^{-1}(T)$ is tangent to $H_\infty$ at 
the point $A_{p,q}$ and $H_{b_{p,q}}\subset\pi^{-1}(T)$ is tangent at 
$B_{p,q}$.  See \fullref{fighorospheres}.
\vspace{6pt}
\begin{figure}[ht!]
  \begin{center}
    \labellist
      \small\hair 2pt
      \pinlabel $H'$ at 83 131
      \pinlabel $H$ at 139 92
      \pinlabel $0=a_{0,0}$ at 168 52
      \pinlabel $b_{0,0}$ at 240 85
      \pinlabel $a_{1,0}$ at 305 52
      \pinlabel $b_{1,0}$ at 377 85
      \pinlabel $a_{2,0}$ at 442 52
      \pinlabel $\C_\infty$ at 550 90
      \pinlabel { { \color{red} $t_\alpha$} } at 227 137
      \pinlabel { { \color{blue} $t_\beta$} } at 170 190
      \pinlabel $A_{0,0}$ at 144 165
      \pinlabel $A_{1,0}$ at 281 165
      \pinlabel $A_{2,0}$ at 418 165
      \pinlabel $B_{0,0}$ at 249 198
      \pinlabel $B_{1,0}$ at 383 198
      \pinlabel $A_{0,1}$ at 205 232
      \pinlabel $A_{1,1}$ at 341 232
      \pinlabel $A_{2,1}$ at 477 232
      \pinlabel $B_{0,1}$ at 310 267
      \pinlabel $B_{1,1}$ at 451 267
      \pinlabel $H_\infty$ at 567 193
    \endlabellist    
    \includegraphics[height=5.3cm]{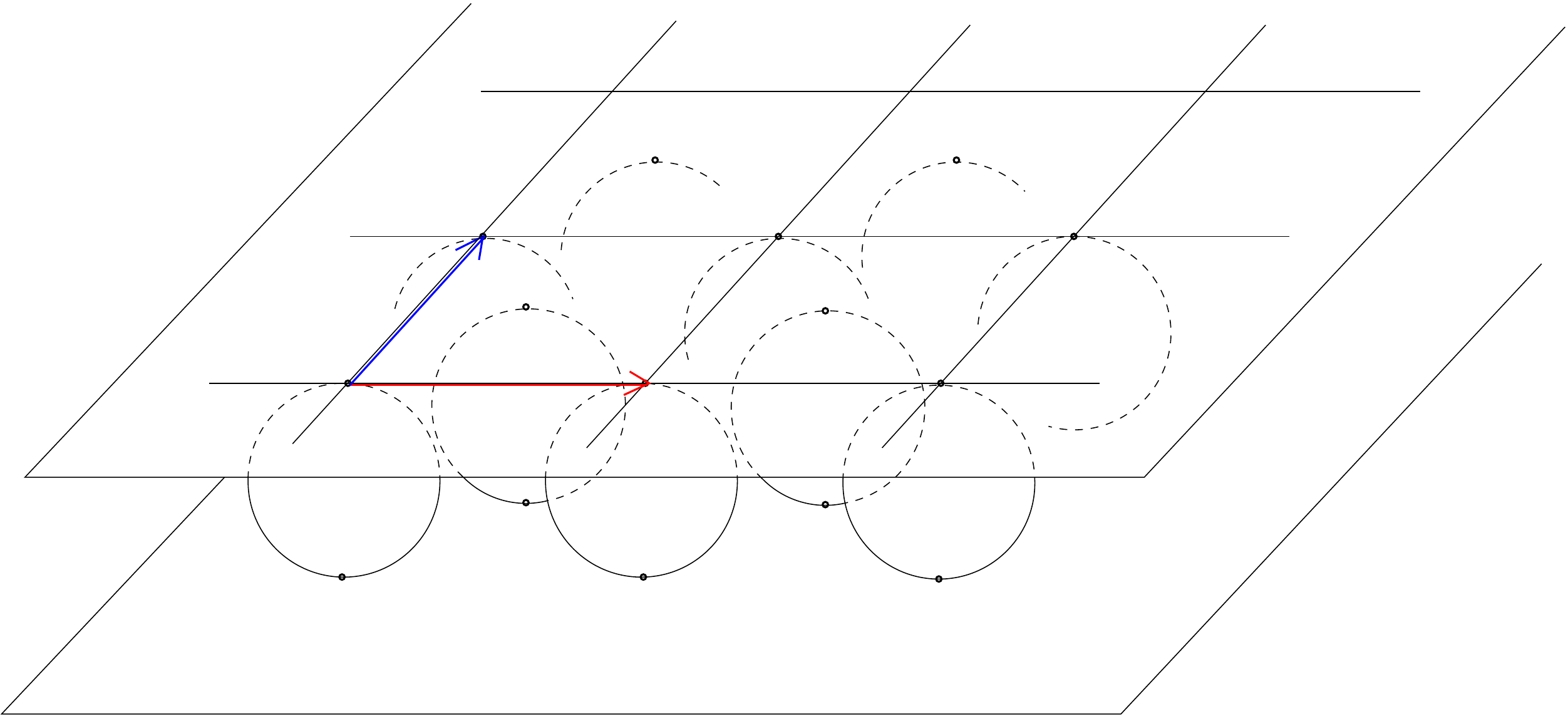}
  \end{center}             
\caption{The normalised preimage of $T$ in the upper half-space model of $\H^3$}
\label{fighorospheres}                   
\end{figure}
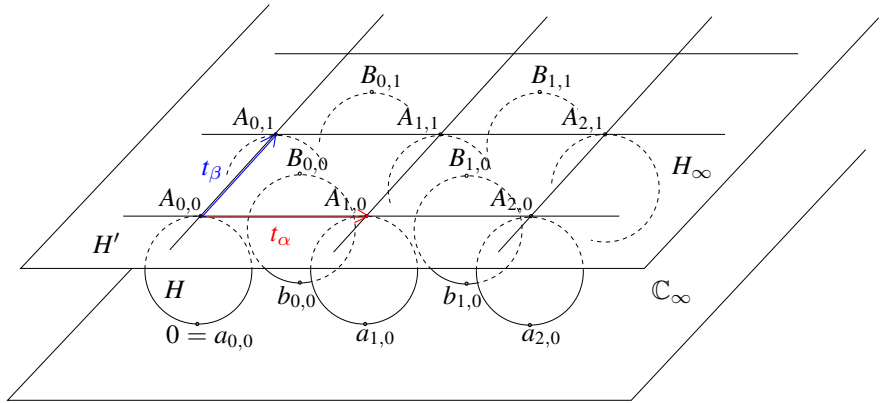
\vspace{6pt}

It is natural to express points on $H_\infty$ by their
coordinates with respect to $\alpha$ and $\beta$. 
Let the point $B_{0,0}$ be given by $b_{0,0}=x_0t_\alpha+y_0t_\beta$,
where $x_0,y_0\in[0,1)$ and $x_0,y_0$ are not both 0, since
$B_{0,0}\neq A_{0,0}$.
Then $a_{p,q}=pt_\alpha+qt_\beta$ and
$b_{p,q}=(p+x_0)t_\alpha+(q+y_0)t_\beta$.
\vspace{6pt}
\begin{defn}
Let $z\in H_\infty$ be at complex position $z=xt_\alpha+yt_\beta$.  We
call $x$ the \emph{$t_\alpha$-coordinate} of $z$ and $y$ the
\emph{$t_\beta$-coordinate} of $z$.
\end{defn}
\vspace{6pt}

The element $g\in\Gamma$ was chosen as an arbitrary isometry
mapping $H=H_0$ to $H'=H_\infty$.  Thus $g$ maps
\begin{itemize}
\item[(i)] 0 to $\infty$ on the sphere at
infinity $\C_\infty$, and
\item[(ii)] bumping point $A_{0,0}$ to some $B_{p,q}$ on $H_\infty$.
\end{itemize}
\vspace{6pt}

Given basis curves $\alpha$ and $\beta$, there is a canonical choice
for $B_{p,q}$ in (ii), namely $B_{0,0}$.
Then, since the M\"obius transformation corresponding to $g$ must map
$0=a_{0,0}\ \mapsto\ \infty\ \mapsto\ b_{0,0}$,
we find that $g$ is of the form       
\begin{equation}
\label{eqgmatrix}
g=\begin{pmatrix} cb_{0,0} & -\frac{1}{c} \\ c & 0 \end{pmatrix}
\qquad\quad\text{for some $c\in\C$\,.}
\end{equation}
Therefore, with this normalisation
the geodesics $\gamma_{p,q}$ are determined by the axes of the
isometries 
\begin{align}
g_{p,q}&=a^pb^qg\notag\\
 &=\begin{pmatrix} 1 & pt_{\alpha}+qt_{\beta} \\
    0 & 1 \end{pmatrix}
    \begin{pmatrix} cb_{0,0} & -\frac{1}{c} \\
    c & 0 \end{pmatrix}\notag\\
 &=\begin{pmatrix} cb_{p,q} & -\frac{1}{c} \\
    c & 0 \end{pmatrix}\label{eqg_pqmatrix}
\end{align}
which map horospheres
$H_0=H_{a_{0,0}}\to H_\infty\to H_{b_{p,q}}$, taking the point
$A_{0,0}$ to $B_{p,q}$.
\newpage

Let $\widetilde{\gamma}_{p,q}\subset\pi^{-1}(\gamma_{p,q})$ be the
axis of this isometry $g_{p,q}$.  Then $\widetilde{\gamma}_{p,q}$
has endpoints $z=z_\pm\in\C_\infty$ satisfying
$${\frac{cb_{p,q}z-\frac{1}{c}}{cz}=z},$$
so that
\begin{align}
z_\pm &= \frac{cb_{p,q}\pm\sqrt{c^2b_{p,q}^2-4}}{2c} \notag\\
 &= \frac{b_{p,q}}{2}\pm
 \sqrt{\left(\frac{b_{p,q}}{2}\right)^2-\frac{1}{c^2}}\label{eqendpts}
\end{align}
Since $|b_{p,q}|=|b_{0,0}+pt_\alpha+qt_\beta|\to\infty$ as
$|(p,q)|\to\infty$, we observe that
\[\Biggl|\sqrt{\left(\frac{b_{p,q}}{2}\right)^2-\frac{1}{c^2}}-\frac{b_{p,q}}{2}\Biggr| \to 0\] as
$|(p,q)|\to\infty$, and hence:
\begin{itemize}
\item[(i)] the endpoints $z_\pm$ satisfy $z_-\to 0$ and
$z_+-b_{p,q}\to 0$, and
\item[(ii)] the radius
$$\Biggl|\sqrt{\left(\frac{b_{p,q}}{2}\right)^2-\frac{1}{c^2}}\Biggr|$$
of the Euclidean semicircle giving the geodesic axis
$\widetilde{\gamma}_{p,q}$ approaches $\infty$,
\end{itemize}
as $|(p,q)|\to\infty$.
\vspace{6pt}

By (ii), for large enough $|(p,q)|$, $\widetilde{\gamma}_{p,q}$
intersects $H_\infty$ in two points, say $C_{p,q}$ closer to $z_-$
and $D_{p,q}$ closer to $z_+$, with complex coordinates $c_{p,q}$ and
$d_{p,q}$ respectively.
Due to (i) and (ii), $c_{p,q}\to a_{0,0}$ and $d_{p,q}-b_{p,q}\to 0$,
so the points $C_{p,q}$ approach $A_{0,0}$ and the points
$g_{p,q}^{-1}(D_{p,q})$ approach $g_{p,q}^{-1}(B_{p,q})=A_{0,0}$.
See \fullref{figgammainHth}.
\vspace{4pt}
\begin{figure}[ht!]
  \begin{center}
    \labellist
      \small\hair 2pt
       \pinlabel $\widetilde{\gamma}_\infty$ at 113 290
       \pinlabel $\widetilde{\gamma}_{p,q}$ at 208 309 
       \pinlabel $g_{p,q}^{-1}(D_{p,q})$ at 44 93
       \pinlabel $A_{0,0}$ at 112 106
       \pinlabel $a_{0,0}=0$ at 110 31 
       \pinlabel $c_{p,q}$ at 156 36
       \pinlabel $t_\beta$ at 136 123
       \pinlabel $C_{p,q}$ at 158 109
       \pinlabel $B_{0,0}$ at 208 133
       \pinlabel $t_\alpha$ at 194 90
       \pinlabel $b_{0,0}$ at 208 57
       \pinlabel $D_{p,q}$ at 458 223
       \pinlabel $B_{p,q}$ at 510 223
       \pinlabel $A_{p,q}$ at 401 198
       \pinlabel $d_{p,q}$ at 474 149
       \pinlabel $b_{p,q}$ at 507 149
       \pinlabel $a_{p,q}$ at 421 120
       \pinlabel $H_\infty$ at 446 85
       \pinlabel $\C_\infty$ at 448 14
     \endlabellist
    \includegraphics[width=\textwidth]{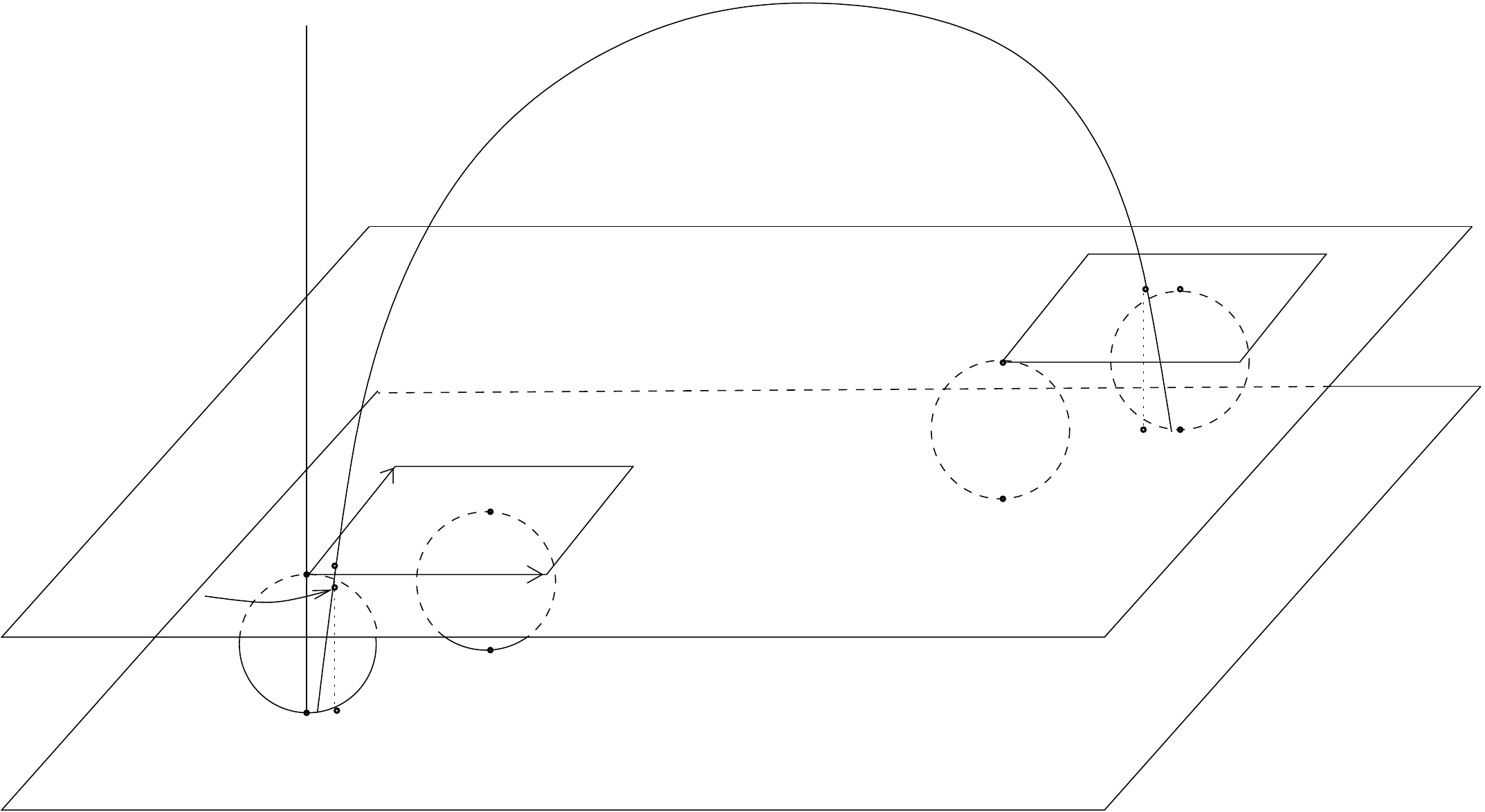}
  \end{center}
  \caption[The geodesic axis $\widetilde{\gamma}_{p,q}$
corresponding to the isometry $g_{p,q}$.]
{The geodesic axis
$\widetilde{\gamma}_{p,q}\subset\pi^{-1}(\gamma_{p,q})$
corresponding to the isometry $g_{p,q}$.
$\widetilde{\gamma}_{p,q}$
can be viewed as approaching the vertical axis
$\widetilde{\gamma}_\infty\subset\pi^{-1}(\gamma_\infty)$
as $|(p,q)|\to\infty$.}
  \label{figgammainHth}
\end{figure}
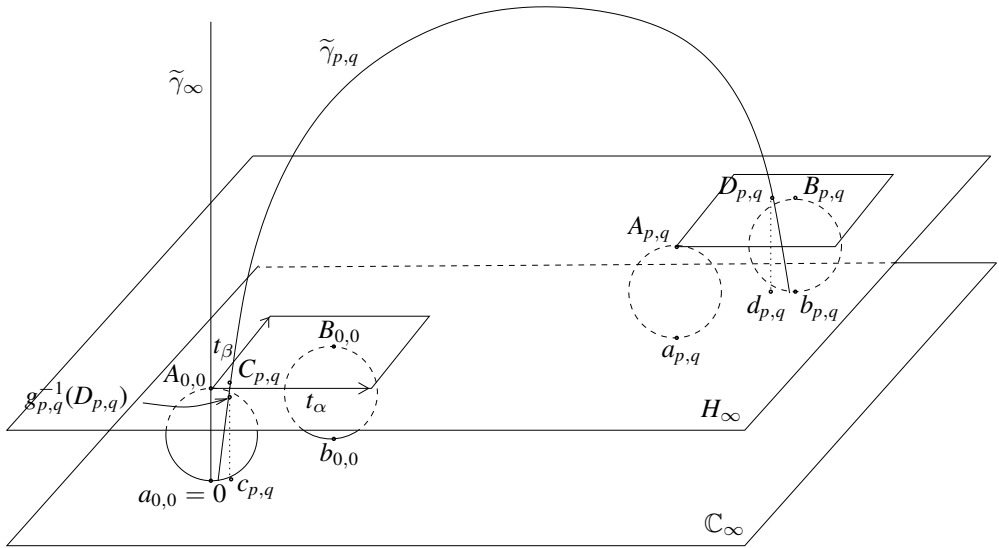
\vspace{4pt}

Therefore,
\begin{equation} \label{eqshortarctoa}
\text{the arc $(g_{p,q}^{-1}(D_{p,q}),C_{p,q})$ approaches the bumping
point $A_{0,0}$,}
\end{equation}
decreasing in length towards 0 as $|(p,q)|\to\infty$.
So for large enough $|(p,q)|$ this arc does not intersect
the preimage $\pi^{-1}(U)$ of the maximal horoball
neighbourhood $U$ of the cusp.
Henceforth we work only with geodesics
$\gamma_{p,q}$ for which $|(p,q)|$ is sufficiently large that this
occurs.
So projecting to $M$, $\gamma_{p,q}$ consists of two distinct arcs---the arc
$\pi([g_{p,q}^{-1}(D_{p,q}),C_{p,q}])$ outside the maximal
cusp neighbourhood $U$, and the arc $\pi((C_{p,q},D_{p,q}))$ inside 
the maximal cusp neighbourhood $U$.
The arc outside $U$ approaches the bumping point $A$, so we call it the 
\emph{short arc} $s_{p,q}$, while the arc inside $U$ increases in length 
towards infinity as $|(p,q)|\to\infty$, so is called the \emph{long arc} 
$l_{p,q}$.
See \fullref{figslarcs} for a schematic picture.
\vspace{4pt}
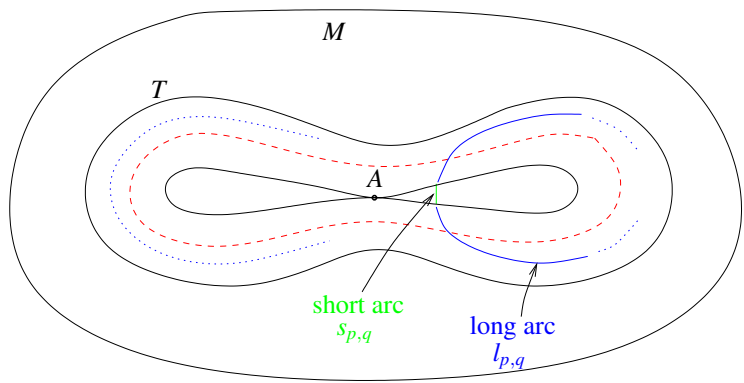
\begin{figure}[ht!]       
  \begin{center}       
    \labellist
      \small\hair 2pt
       \pinlabel $M$ at 190 205
       \pinlabel $T$ at 88 171
       \pinlabel $A$ at 214 119
       \pinlabel { { \color{green} short arc} } at 204 46
       \pinlabel { { \color{blue} long arc} } at 294 32
       \pinlabel { { \color{green} $s_{p,q}$ } } at 204 29
       \pinlabel { { \color{blue} $l_{p,q}$ } } at 294 16
     \endlabellist
    \includegraphics[height=5cm]{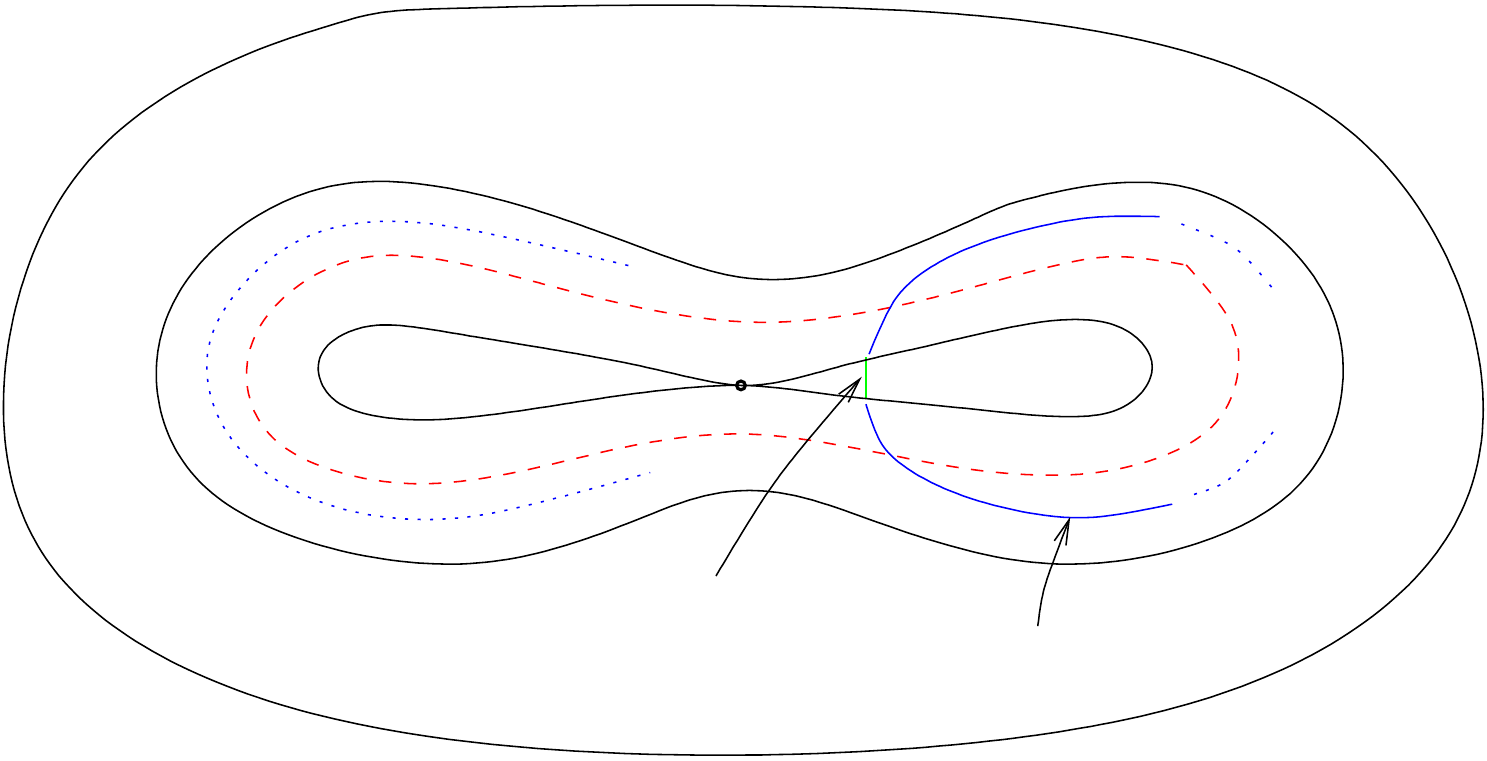}
  \end{center}             
  \caption{A schematic view of the short and long arcs of a geodesic
$\gamma_{p,q}$}       
  \label{figslarcs}   
\end{figure}
\newpage

So $\gamma_{p,q}$ is simple precisely when its two arcs $s_{p,q}$ and
$l_{p,q}$ are embedded in $M$.
We show that this holds for an infinite subfamily of these geodesics.


\section{An infinite subfamily is simple}
\label{secsimplesubfamily}

Let $\widetilde{s}_{p,q}$ denote the lift
$[g_{p,q}^{-1}(D_{p,q}),C_{p,q}]$ of the short arc $s_{p,q}$, and
$\widetilde{l}_{p,q}$ the lift $(C_{p,q},D_{p,q})$ of the long arc
$l_{p,q}$, shown in \fullref{figgammainHth}.
\begin{defn}
\label{defrclose}
Let $R$ be the region in $\H^3$ consisting of those points which are
closer to the point $A_{0,0}$ than to any other lift of the bumping
point $A$.
Say that the short arc $s_{p,q}$ is \emph{$R$--close} to $A$ if its lift
$\widetilde{s}_{p,q}$ in $\H^3$ is contained in $R$.
\end{defn}
By its definition, the projection of $R$ to $M$ is injective, and so
$s_{p,q}$ is embedded in $M$ if it is $R$--close to $A$.

\begin{lem} \label{lemshortsimple}
In any cusped hyperbolic 3--manifold $M$ all but at most finitely many
geodesics $\gamma_{p,q}$ have short arcs which are $R$--close to $A$
and hence embedded.
\end{lem}
\begin{proof}
By its definition, the region $R$ contains a ball
$B_\epsilon(A_{0,0})$ of radius $\epsilon$ around $A_{0,0}$ for some
$\epsilon>0$.  By \eqref{eqshortarctoa}, $\widetilde{s}_{p,q}$
approaches $A_{0,0}$ as $|(p,q)|\to\infty$, so there is some
$N_0\in\N$ such that if $|(p,q)|>N_0$, $\widetilde{s}_{p,q}\subset
B_\epsilon(A_{0,0})\subset R$.  Then $s_{p,q}$ is $R$--close to $A$ and
hence embedded.
\end{proof}

Consider now the long arcs.
Isotoping $l_{p,q}$ radially from the cusp to a  
curve $L_{p,q}\subset T$ corresponds to projecting the lift       
$\widetilde{l}_{p,q}$ vertically downwards to a line segment, say       
$\widetilde{L}_{p,q}$ on the horosphere $H_\infty$, as in       
\fullref{figl_pq}.       
\begin{figure}[ht!]     
  \begin{center}       
    \labellist
      \small\hair 2pt
       \pinlabel $A_{0,0}$ at 71 39
       \pinlabel $t_\beta$ at 101 79
       \pinlabel $t_\alpha$ at 135 41
       \pinlabel $C_{p,q}$ at 159 116
       \pinlabel { { \color{blue} $\widetilde{L}_{p,q}$ } } at 336 173
       \pinlabel $H_\infty$ at 440 71
       \pinlabel $D_{p,q}$ at 525 266
       \pinlabel $B_{p,q}$ at 547 219
     \endlabellist
    \includegraphics[height=4.8cm]{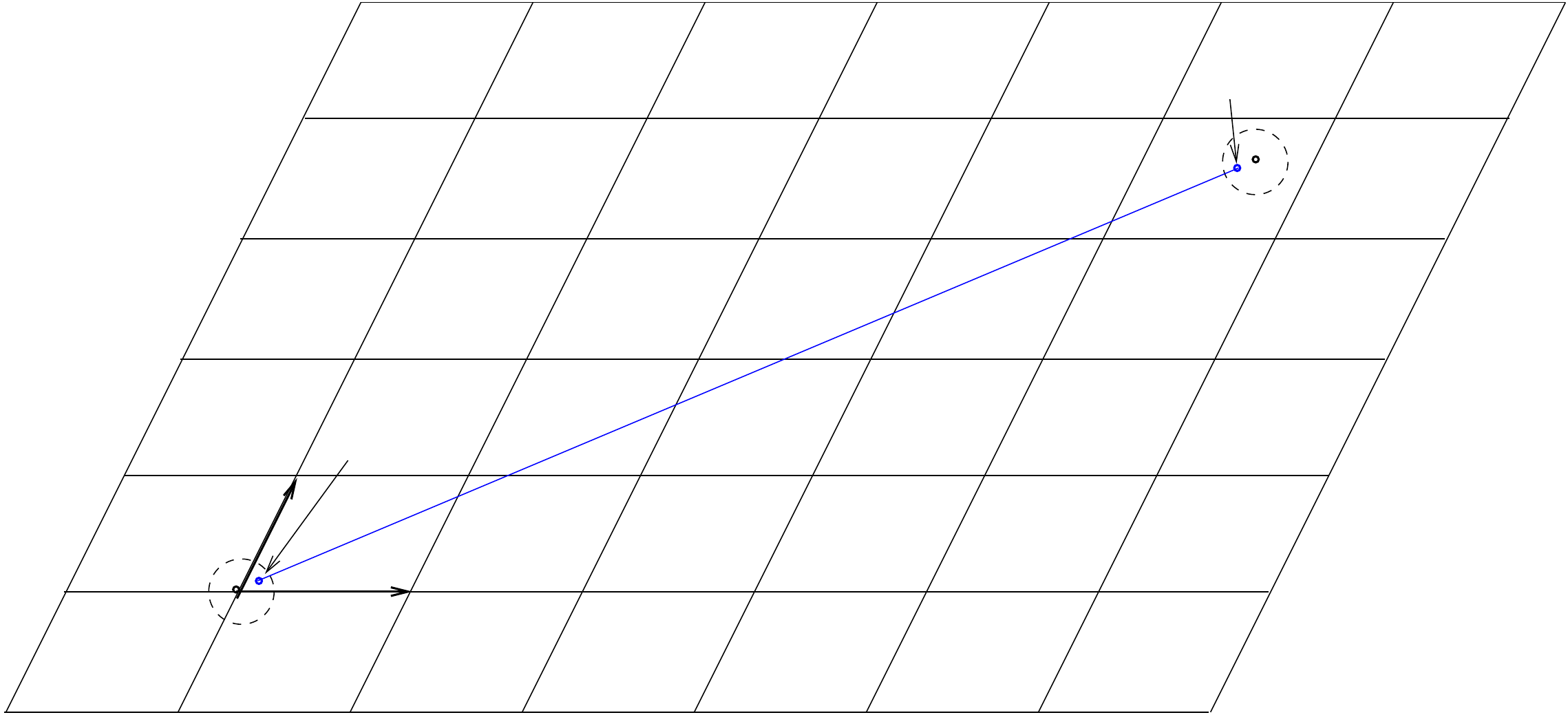}
  \end{center}
  \caption{A lift $\widetilde{L}_{p,q}$ of the radial projection of
$l_{p,q}$ to $T$}
  \label{figl_pq}
\end{figure}
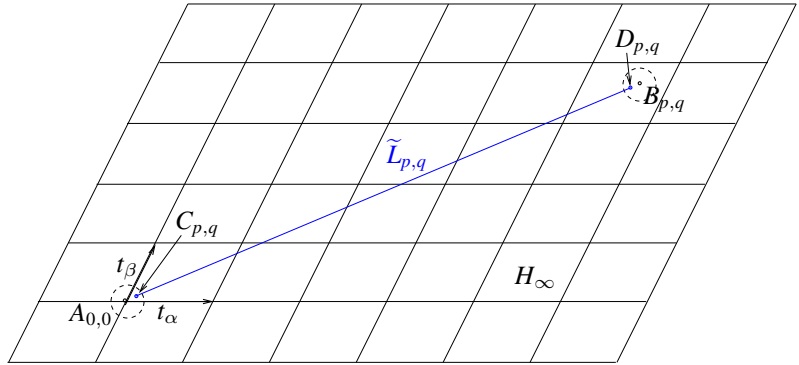

\begin{lem} \label{lemlpqnonsimple}
The long arc $l_{p,q}$ of a geodesic $\gamma_{p,q}$ is nonsimple if
and only if the projected line segment $\widetilde{L}_{p,q}$ contains
two distinct points which differ by a translation by an integer linear
combination of $t_\alpha$ and $t_\beta$.
\end{lem}
\begin{proof}
Clearly if $l_{p,q}$ has self-intersection the two points on
$\widetilde{L}_{p,q}$ giving this self-intersection must project to
the same point on $T$, and so differ by an integer linear combination of
$t_\alpha$ and $t_\beta$.
Conversely, if any two distinct points on $\widetilde{L}_{p,q}$ differ
by an integer linear combination of $t_\alpha$ and $t_\beta$, then so do two
points equidistant from the midpoint of $\widetilde{L}_{p,q}$, by
translation along the line segment.  Such equidistant points are
images of points at the same vertical height on the arc
$\widetilde{l}_{p,q}$ and thus the same distance from $T$ on
$l_{p,q}$.  These points therefore coincide in $M$, producing a
self-intersection of $l_{p,q}$.
\end{proof}

We are now ready to prove \fullref{thminfinite} which we recall
here:
\begin{cuspedthm} 
Every cusped orientable hyperbolic 3--manifold contains infinitely
many simple closed geodesics.
\end{cuspedthm}
\begin{proof}
By \fullref{lemshortsimple} it suffices to show that infinitely
many geodesics $\gamma_{p,q}$ have embedded long arcs.

Recall that $b_{0,0}=x_0t_\alpha+y_0t_\beta$, where $x_0,y_0\in
[0,1)$ and at least one of $x_0$ and $y_0$ is nonzero.
Assume without loss of generality that $y_0>0$---that is, the
$t_\beta$--coordinate of $b_{0,0}$ is positive.

Consider then the subfamily of geodesics given by $q=0$, ie those
given by elements of the form $g_{p,0}=a^pg$.  We know that $|c_{p,0}|
=\dist_{H_\infty}(A_{0,0},C_{p,0})=\dist_{H_\infty}(B_{p,0},D_{p,0})\to 0$
as $|p|\to\infty$.  Let $\delta=\min\{\frac{1}{2}y_0,\frac{1}{2}(1-y_0)\}$.
Then there exists $P\in\N$ such that for all $p$ with $|p|>P$,
$|c_{p,0}|<\delta$.
So for $|p|>P$, the complex number $d_{p,0}-c_{p,0}$ giving the
translation along $\widetilde{L}_{p,0}$ has $t_\beta$--coordinate in
the interval $(y_0-2\delta,y_0+2\delta)\subset (0,1)$.  Hence when
$|p|>P$, $\widetilde{L}_{p,0}$ cannot contain points differing by an
integer linear combination of $t_\alpha$ and $t_\beta$, and so
by \fullref{lemlpqnonsimple} the long arc $l_{p,0}$ is simple and
we are done.
\end{proof}

Note that the above result cannot be strengthened to show that all but finitely
many $\gamma_{p,q}$ are embedded.
For example, it is shown in~\cite{phd} that
in the figure-eight knot complement infinite families of nonsimple 
$\gamma_{p,q}$ occur.


We conclude by remarking that our proof also gives a description of the isotopy
classes of the simple geodesics $\gamma_{p,q}$, for $|(p,q)|$ sufficiently 
large.
They are the union of a long arc $l_{p,q}$ spiralling inside       
the maximal cusp neighbourhood $U$,       
and a short arc $s_{p,q}$ outside this cusp neighbourhood.       
The long arc can be isotoped to a       
Euclidean line segment $L_{p,q}$ on $T$, the (singular) torus boundary
of $U$,
while $s_{p,q}$ is a short connecting arc between the endpoints of
$L_{p,q}$, close to the singular point $A$ of $T$.       
See \fullref{figgammapq} for a schematic view of this description.
\begin{figure}[ht!]       
 \begin{center}       
    \labellist
      \small\hair 2pt
       \pinlabel $T$ at 489 161
       \pinlabel $A$ at 229 114
       \pinlabel { { \color{blue} $L_{p,q}$ } } at 240 196
       \pinlabel { { \color{green} $s_{p,q}$ } } at 340 99
       \pinlabel { { \color{red} cusp } } at 544 74
     \endlabellist
  \includegraphics[height=4cm]{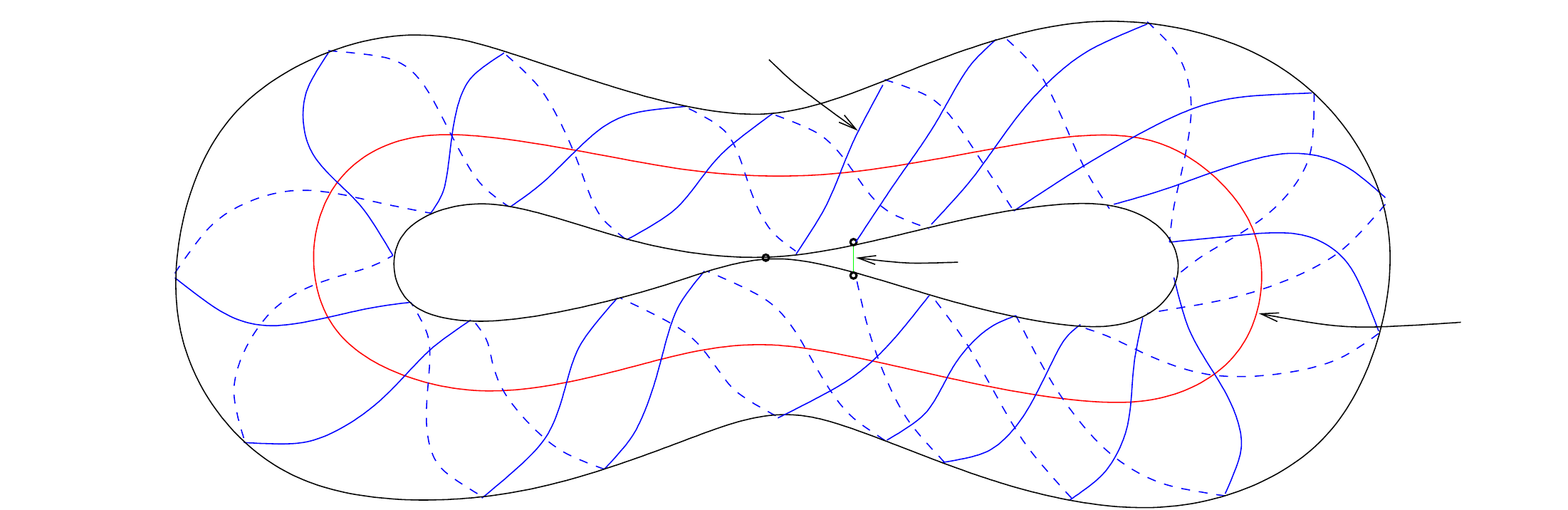}
 \end{center}                 
 \caption[The isotopy class of a geodesic $\gamma_{p,q}$.]
 {The isotopy class of a geodesic $\gamma_{p,q}$ can be              
described as the union of a long arc isotoped to lie on the singular  
torus $T$ and a short connecting arc outside the maximal cusp       
neighbourhood.}
 \label{figgammapq}       
\end{figure}
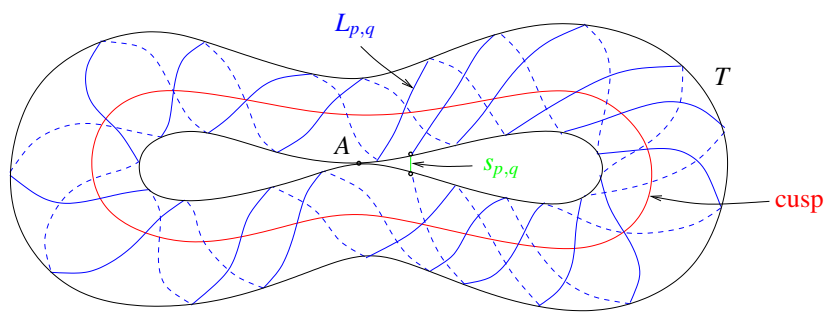         

In~\cite{phd} we investigate this further and obtain precise topological 
descriptions for the isotopy classes of infinite subfamilies of the geodesics
$\gamma_{p,q}$ in any cusped hyperbolic 3--manifold.
We also draw explicit projection diagrams for these geodesics in the case of 
the figure-eight knot complement.


\bibliographystyle{gtart}
\bibliography{link}

\end{document}